\documentclass[a4paper,11pt]{article}
\usepackage{amsfonts}
\usepackage{amsthm}
\usepackage{amssymb}
\usepackage{latexsym}
\usepackage{graphics}
\usepackage{graphicx}
\usepackage{enumerate}
\usepackage{amsmath}
\usepackage{verbatim}
\usepackage{hyperref}
\usepackage{bibspacing}

\newtheorem{theorem}{Theorem}[section]

\newtheorem{proposition}[theorem]{Proposition}

\theoremstyle{definition}

\newtheorem{remark}[theorem]{Remark}
\newtheorem{definition}[theorem]{Definition}

\numberwithin{equation}{section}



\newcommand{\vol}{\textrm{Vol}}
\newcommand{\Var}{\textrm{Var}}
\newcommand{\Ent}{\textrm{Ent}}
\newcommand{\norm}[1]{\left\Vert#1\right\Vert}

\newcommand{\abs}[1]{\left\vert#1\right\vert}

\newcommand{\brac}[1]{\left(#1\right)}

\newcommand{\R}{\mathbb{R}}

\newcommand{\F}{\mathcal{F}}

\newcommand{\eps}{\epsilon}

\newcommand{\Id}{{\rm Id}}

\newlength{\defbaselineskip}
\newcommand{\setlinespacing}[1]           {\setlength{\baselineskip}{#1 \defbaselineskip}}

\renewcommand*{\thefootnote}{\fnsymbol{footnote}}

\begin{document}

\title{Reverse H\"{o}lder Inequalities for log-Lipschitz Functions}

\medskip

\author{Emanuel Milman\thanks{Department of Mathematics, Technion-Israel Institute of Technology, Haifa 32000, Israel. Email: emilman@tx.technion.ac.il. }}

\begingroup    \renewcommand{\thefootnote}{}           \footnotetext{The research leading to these results is part of a project that has received funding from the European Research Council (ERC) under the European Union's Horizon 2020 research and innovation programme (grant agreement No 637851).}

\date{} 
\maketitle

\begin{abstract}
Reverse H\"{o}lder inequalities for a class of functions on a probability space constitute an important tool in Analysis and Probability. After revisiting how a (modified) log-Sobolev inequality can be used to derive reverse H\"{o}lder inequalities for the class of log-Lipschitz functions, we obtain a weaker condition using general Transport-Entropy inequalities, which can also handle approximately log-Lipschitz functions. In its weakest form, the condition degenerates to the assumption of satisfying a concentration inequality. We compare this with a scenario in which the underlying space only satisfies a Poincar\'e inequality. 
\end{abstract}

\section{Introduction}

Let $(\Omega,d,\mu)$ denote a metric-measure space, namely a complete separable metric space $(\Omega,d)$ endowed with a Borel measure $\mu$. In this note, we will always assume that the measure $\mu$ is a \emph{probability} measure. By H\"{o}lder's (or Jensen's) inequality we have:
\begin{equation} \label{eq:intro-Jensen}
q < p \;\; \Rightarrow \;\; \norm{f}_{L^q(\mu)} \leq \norm{f}_{L^p(\mu)} ,
\end{equation}
for all functions $f$ for which the above integrals make sense. We are interested in finding conditions on the space which ensure the validity of the following reverse H\"{o}lder inequalities:
\begin{equation} \label{eq:intro-RH}
 a < q < p < b ~,~ f \in \F  \;\; \Rightarrow \;\; \norm{f}_{L^p(\mu)} \leq C_{\F,q,p} \norm{f}_{L^q(\mu)} ,
\end{equation}
for an appropriate class of functions $\F$ and range parameters $a < b$. Such reverse H\"{o}lder inequalities are in some sense a manifestation of the \emph{concentration} properties of the space. As such, they constitute an important tool in Analysis and Probability, and have found diverse applications in these disciplines (see below).  

Before proceeding, we remark that the class $\F_{1}$ of $1$-Lipschitz functions on $(\Omega,d)$, which is often used to state and study concentration properties of $(\Omega,d,\mu)$, 
is not well-suited for obtaining the reverse H\"{o}lder inequalities (\ref{eq:intro-RH}). Indeed, the class of functions satisfying (\ref{eq:intro-RH}) is a cone, clearly invariant under multiplication $f \mapsto \alpha f$, whereas the class of $1$-Lipschitz functions is not preserved under multiplicative operations, but rather additive ones. As such, a typical concentration result for $1$-Lipschitz functions $f$ is a statement of the form:
\[
\norm{f - \int f d\mu}_{L^p(\mu)} \leq C_p ,
\]
which yields an \emph{additive}, rather than a multiplicative, reverse H\"{o}lder inequality for $1$-Lipschitz functions:
\[
\norm{f}_{L^p(\mu)} \leq \norm{f}_{L^1(\mu)} + C_p . 
\]
Consequently, a natural class of functions for obtaining (\ref{eq:intro-RH}) is that of $L$-\emph{log-Lipschitz} functions $\F_{\log,L}$, defined as those positive functions $f$ on $\Omega$ so $\log f$ is $L$-Lipschitz:
\[
|\log f (x) - \log f(y)| \leq L d(x,y) \;\;\; \forall x,y \in \Omega . 
\]

Let us now give several well-known examples of spaces $(\Omega,d,\mu)$ and classes of functions $\F$ for which (\ref{eq:intro-RH}) holds (which by no means constitutes a comprehensive list). We refer to the next section for missing definitions. 
\begin{itemize}
\item If $(\Omega,d,\mu)$ satisfies a log-Sobolev inequality with constant $\lambda_{LS}$, it is known by the Herbst argument \cite[Proposition 5.4.1]{BGL-Book} that (\ref{eq:intro-RH}) holds for all $L$-log-Lipschitz functions and all $-\infty < q < p < \infty$ with $C_{\F_{log,L},q,p} = \exp(\frac{L^2 (p-q)}{2 \lambda_{LS}})$. In particular, it holds that:
\begin{equation} \label{eq:intro-RH0}
f \in \F_{\log,L} \;\; , \;\; p > 0 \;\; \Rightarrow \;\; \norm{f}_{L^p(\mu)} \leq C_{L,p} \norm{f}_{L^0(\mu)}  ,
\end{equation}
with $C_{L,p} = \exp\brac{\frac{L^2 p}{2 \lambda_{LS}}}$, where of course $\norm{f}_{L^0(\mu)}$ is defined in the limiting sense as $\exp ( \int \log f \, d\mu)$. See also Bobkov--G\"{o}tze \cite{BobkovGotzeLogSobolev} for an extension to the case when $f$ is assumed to be log-Lipschitz only in some averaged sense (and compare with the formulation of Theorem \ref{thm:Lb}). \item When $\Omega$ is a smooth connected oriented Riemannian-manifold-with-boundary $(M,g)$ and $\mu$ is a probability measure with smooth and positive density on $\Omega$, various conditions ensure the validity of a reverse H\"{o}lder inequality for eigenfunctions $f$ of the weighted Neumann Laplacian $\Delta_{g,\mu}$ (with vanishing Neumann boundary conditions $f_\nu|_{\partial M} \equiv 0$). For example, it was shown by Gross \cite[Theorem 5.2.3]{BGL-Book} that the validity of the log-Sobolev inequality is equivalent to the hypercontractivity of the heat semi-group generated by $\Delta_{g,\mu}$; using this, it is immediate to show that \cite[Section 5.3]{BGL-Book}:
\[
\begin{cases} -\Delta_{g,\mu} f = \lambda f \\ f_\nu|_{\partial M} \equiv 0 \end{cases}  , \;\; 1 < q < p <\infty \;\; \Rightarrow \;\; \norm{f}_{L^p(\mu)} \leq \brac{\frac{p-1}{q-1}}^{\frac{\lambda}{2 \lambda_{LS}}} \norm{f}_{L^q(\mu)} .
\]
When moreover the space satisfies a dimensional Sobolev--Gagliardo--Nirenberg inequality \cite[Chapter 6]{BGL-Book}, which is known to be equivalent to boundedness of the heat-kernel (or ultracontractivity of the heat semi-group) \cite[Section 6.3]{BGL-Book}, the above estimates can further be extended to the case $p=\infty$ (but with dimension-dependent constants). In particular, dimension-free reverse H\"{o}lder inequalities have been obtained in \cite{ReverseHolderOnSphere} for spherical harmonics and hence polynomials on the $n$-sphere. More recently, Cianchi--Maz'ya \cite{CianchiMazya-ReverseHolderForEigenfunctions} obtained a precise characterization, in terms of the space's isoperimetric and isocapacitary behaviour, of when a reverse H\"{o}lder inequality with $p < \infty$ and $p=\infty$ holds for the eigenfunctions. 
\item In the discrete setting, an analogous reverse H\"{o}lder inequality is valid for eigenfunctions of the discrete Laplacian whenever hypercontractivity holds. It is worthwhile to note that this observation in the discrete setting, and specifically for the uniform measure $\mu$ on the Hamming cube $\{-1,1\}^n$, most probably predates the one in the continuous setting mentioned above. Indeed, reverse H\"{o}lder inequalities for homogeneous polynomials $f : \{-1,1\}^n \rightarrow \R$ of degree $k$ were obtained by Bonami as early as \cite{Bonami-LambdaPOnCube,Bonami-PhDInJournal}; this was later understood to be a consequence of hypercontractivity by Nelson \cite{Nelson-Hypercontractivity} -- see Exercises 9.37, 9.38 and the Notes at the end of Chapter 9 of \cite{ODonnell-Book}. These observations motivated the foundational work by Kahn--Kalai--Linial \cite{KKL-Influence}, and have had a profound influence on the fields of Combinatorics and Theoretical Computer Science. 
\item When $\Omega$ is a bounded domain in $\R^n$ and $f$ is the associated ground state, namely the eigenfunction of the Laplacian with vanishing Dirichlet boundary conditions corresponding to the first (positive) eigenvalue $\lambda$, various authors have obtained increasingly general and sharp reverse H\"{o}lder inequalities for $f$ (with the Euclidean ball having the same eigenvalue $\lambda$ being the extremal case) -- see the works by Payne--Rayner \cite{PayneRayner1,PayneRayner2}, Kohler-Jobin \cite{KohlerJobin-ReverseHolderForGroundState}  and Chiti \cite{Chiti-ReverseHolderForDirichletEigenfunctions}. \item When $\mu = \exp(-V) dx$ is a log-concave probability measure on $\R^n$, meaning that $V : \R^n \rightarrow \R \cup \{+\infty\}$ is convex, a well-known consequence of Borell's lemma \cite{Borell-logconcave} ensures that for any semi-norm $f$ on $\R^n$, one has:
\[
1 \leq q < p < \infty \;\; \Rightarrow \;\; \norm{f}_{L^p(\mu)} \leq C \frac{p}{q} \norm{f}_{L^q(\mu)} . 
\]
This was extended to all $q \geq 0$ by Lata{\l}a \cite{LatalaZeroMomentKhinchine} and to $q > -1$ by Gu\'edon \cite{Guedon-extension-to-negative-p}. 
\item When $\mu$ is again a log-concave probability measure on $\R^n$ and $f$ is a degree $d$ polynomial, reverse H\"{o}lder inequalities for $f$ have been obtained 
in the range $-1/d < q < p < \infty$, starting from the work of  Bourgain \cite{Bourgain-LK}, and extended using the localization technique by Bobkov \cite{BobkovPolynomials}, Carbery--Wright \cite{CarberyWrightPolynomials} and Nazarov--Sodin--Volberg \cite{NazarovSodinVolbergPolynomials}; for more on this topic, we refer to the excellent survey of Fradelizi \cite{FradeliziUltimateKhinchine}, which also extends these results further.
\item When $\mu$ is still a log-concave probability measure which is in addition assumed to be isotropic, reverse H\"{o}lder inequalities have been obtained for the Euclidean norm $f(x) = |x|$ in increasing degree of precision by Paouris \cite{Paouris-IsotropicTail,PaourisSmallBall}, Klartag \cite{KlartagCLP,KlartagCLPpolynomial,KlartagUnconditionalVariance}, Fleury \cite{FleuryImprovedThinShell}, Gu\'edon--Milman \cite{GuedonEMilmanInterpolating} and Lee--Vempala \cite{LeeVempala-KLS}. 
\end{itemize}

The starting point of this work is the first example above, which 
provides a satisfactory condition for ensuring the validity of (\ref{eq:intro-RH}) for the class of log-Lipschitz functions $\F_{\log}$ -- a log-Sobolev inequality. However, on many natural spaces, a log-Sobolev inequality does not hold, or alternatively holds with very bad constant. As a prototypical example, consider the case of the two-sided exponential probability measure $\nu = \frac{1}{2} \exp(-|x|) dx$ on $\R$. It is well-known that $(\R^n,|\cdot|,\nu^{\otimes n})$ does not satisfy any log-Sobolev inequality (since the latter necessarily yields sub-Gaussian tails), but one may wonder whether (\ref{eq:intro-RH}) still holds for the class of $L$-log-Lipschitz functions $\F_{\log,L}$ in a certain range $- b < q < p < b$. If so, it is clear that necessarily $b \leq 1/L$, since $\norm{f}_{L^p(\nu^{\otimes n})} < \infty$ if and only if $|p| < 1/L$ for the $L$-log-Lipschitz function $f(x) = \exp(L x_1)$. 

As a warm up, we start by providing an answer to the latter question for general spaces satisfying a Poincar\'e inequality. As is known to experts, the Herbst argument can be applied whenever a \emph{modified} log-Sobolev inequality holds. Instead of demonstrating this in full generality, we restrict our attention to the original modified log-Sobolev inequality introduced by Bobkov--Ledoux \cite{BobkovLedouxModLogSobAndPoincare}, which was shown by Bobkov--Gentil--Ledoux \cite{BobkovGentilLedoux} to be equivalent to a Poincar\'e inequality. It is well-known that $(\R^n,|\cdot|,\nu^{\otimes n})$ satisfies a Poincar\'e inequality with sharp constant $\lambda_1 = \frac{1}{4}$ (e.g. \cite[(4.4.3)]{BGL-Book}). 

\begin{theorem}[following Bobkov-Ledoux \cite{BobkovLedouxModLogSobAndPoincare}] \label{thm:intro-Poincare}
Assume that $(\Omega,d,\mu)$ satisfies the following Poincar\'e inequality:
\[
\int g d\mu = 0 \;\; \Rightarrow \;\; \lambda_1 \int g^2 d\mu \leq \int |\nabla g|^2 d\mu ,
\]
for all locally Lipschitz functions $g$. Then for any $L$-log-Lipschitz function $f$:
\[
  - 2 \sqrt{\lambda_1} / L < q < p < 2 \sqrt{\lambda_1} / L  \;\; \Rightarrow \;\; \norm{f}_{L^p(\mu)} \leq C_{L,q,p} \norm{f}_{L^q(\mu)}  ,
\]
with:
\begin{equation} \label{eq:K}
C_{L,q,p} = \exp\brac{L^2 \int_{q}^{p} K(|t| L) dt } ~,~ K(\ell) := \frac{1}{2 \lambda_1} \brac{\frac{ 2 \sqrt{\lambda_1} + \ell }{2 \sqrt{\lambda_1} - \ell}}^2 \exp( \ell \sqrt{5} / \lambda_1) . 
\end{equation}
\end{theorem}

A key feature of using the Herbst argument is that the resulting constant $C_{L,q,p}$ in the reverse H\"{o}lder inequality will satisfy $\lim_{p \rightarrow q+} C_{L,q,p} = 1$. It is then natural to wonder whether, if one is willing to give up on this requirement, it is enough to assume a weaker condition than a modified log-Sobolev inequality. In particular, we are interested in conditions which ensure the weaker (\ref{eq:intro-RH0}), comparing the $p$-th moment ($p > 0$) to the zeroth one only (thereby yielding a loose comparison to all $q$-th moments for $0 < q < p$ by Jensen's inequality (\ref{eq:intro-Jensen})). 

One very simple such condition is given by sub-exponential concentration inequalities for $1$-Lipschitz functions $g$, namely an assumption of the form:
\[
\exists \mathcal{K} : \R_+ \rightarrow \R_+ \;\;\; \forall g \in \F_1 \;\;\; \forall s > 0 \;\;\; \mu\{ x \in \Omega \; ; \; g(x) \geq \int g d\mu + s \} \leq \mathcal{K}(s) ,
\]
where $\mathcal{K}$ is a function decaying to zero exponentially or faster. 
Indeed, given an $L$-log-Lipschitz $f$, simply apply the above to the $1$-Lipschitz function $\frac{1}{L} \log f$, yielding:
\[
\mu\{ x \in \Omega ; f(x) \geq t \norm{f}_{L^0(\mu)} \} = \mu\{ x \in \Omega ; \frac{1}{L} \log f(x) \geq  \int \frac{1}{L} \log f  d\mu + \frac{1}{L} \log t  \} \leq \mathcal{K}(\frac{1}{L} \log t ) ,
\]
which may be immediately integrated (since $\mathcal{K}$ decays sub-exponentially) to obtain a reverse H\"{o}lder inequality for $f$.

The purpose of this note is to present a very general condition on $(\Omega,d,\mu)$, weaker than having a (modified) log-Sobolev inequality but also stronger than having a sub-exponential concentration inequality, which ensures the validity of (\ref{eq:intro-RH0}) for a rather general class of functions $f$ with good control over $|\nabla \log f|$. 
The condition is formulated in great generality, covering all possible Transport-Entropy inequalities, which in their weakest form degenerate to concentration inequalities as above. 
It turns out that for the mere purpose of obtaining (\ref{eq:intro-RH0}) for the class of $L$-log-Lipschitz functions, we do not know how to exploit any information beyond concentration properties of Lipschitz functions. As expected, whenever a strictly sub-exponential concentration inequality holds, then all moments ($L^p(\mu)$-norms) of log-Lipschitz functions are comparable; on the other hand,  if the space enjoys only exponential concentration, then the $p$-moments will be comparable for small enough $|p|$ only (in accordance with the obvious obstruction that the $L^p(\mu)$-norm of a log-Lipschitz function can be infinite if $|p|$ is greater than the magnitude of the exponential concentration).

However, our more general analysis has two advantages: first, it confirms (quantitatively) that $C_{L,p} \rightarrow 1$ as $p \rightarrow 0+$ in (\ref{eq:intro-RH0}) if a \emph{tight} concentration inequality is available; and second, when $\log f$ is only assumed approximately Lipschitz from one side:
\[
\log f(y) \geq \log f(x) - L(x) d(x,y) - b(x) \;\;\; \forall x,y \in \Omega,
\]
for appropriate Borel functions $L,b$ with \emph{unbounded} $L$, a concentration inequality does not seem to be enough to derive (\ref{eq:intro-RH0}), and we need to use the added power of the Transport-Entropy inequality. We defer formulating our precise results to Section \ref{sec:results}, after introducing the appropriate notation in Section \ref{sec:notation}.

As the expert reader will surely note, the proofs of these results are extremely simple once the appropriate notation and background is recalled, and very similar (if not identical) arguments are well-known in the literature. However, the author has not seen these statements explicitly elsewhere in the literature, and thought that it would be a good idea to record them in this unpresumptuous note. 

\section{Notation and Background} \label{sec:notation}

In this section, we fix some notation and terminology which will be used for stating our main result in the next section. For additional background on Functional, Transport-Entropy and Concentration inequalities, we refer to the excellent monographs by Ledoux \cite{Ledoux-Book}, Bakry--Gentil--Ledoux \cite{BGL-Book} or the paper \cite{EMilmanGeometricApproachPartII}. 

\subsection{Functional Inequalities}

Let $\F_{loc} = \F_{loc}(\Omega,d)$ denote the space of functions which are Lipschitz on every ball in $(\Omega,d)$, and let $f \in \F_{loc}$. Functional inequalities compare between some type of expression measuring the $\mu$-averaged oscillation of $f$, and an expression measuring the $\mu$-averaged magnitude of the gradient $|\nabla f|$. Here $|\nabla f|$ is defined as the following Borel function:
\[
 \abs{\nabla f}(x) := \limsup_{d(y,x) \rightarrow 0+} \frac{|f(y) - f(x)|}{d(x,y)} 
\]
(and 0 if $x$ is an isolated point - see \cite[pp. 184,189]{BobkovHoudre} for more details). Of course, when $(\Omega,d)$ is a smooth Riemannian manifold $(M,g)$ with its induced geodesic distance $d$, we have $\abs{\nabla f}:= \sqrt{g(\nabla f,\nabla f)}$.

A prime example of functional inequalities which have revolutionized the fields of Analysis and PDE is given by the family of Sobolev-Gagliardo-Nirenberg inequalities \cite[Chapter 6]{BGL-Book}. In this work, we will mostly emphasize the relation to the Poincar\'e and log-Sobolev inequalities.

\begin{definition}
$(\Omega,d,\mu)$ is said to satisfy a Poincar\'e inequality with constant $\lambda_1 > 0$ if:
\[
\lambda_1 \Var_{\mu}(f) \leq \int |\nabla f|^2 d\mu \;\;\; \forall f \in \F_{loc} ~.
\]
Here $\Var_\mu(f) := \int (f - \int f d\mu)^2 d\mu$ denotes the variance of $f$.  
\end{definition}

When $\Omega$ is a smooth connected oriented manifold-with-boundary $(M,g)$ and $\mu = \exp(-V) \vol_g$ is a probability measure with smooth and positive density on $\Omega$ with respect to the Riemannian volume measure $\vol_g$, the best constant $\lambda_1>0$ above coincides with the first non-zero eigenvalue of the associated weighted Laplacian  $-\Delta_{g,\mu}$, given by $\Delta_{g,\mu} f := \Delta_g f - g(\nabla V , \nabla f)$,  with vanishing Neumann boundary conditions on $\partial M$.

\begin{definition}
$(\Omega,d,\mu)$ is said to satisfy a log-Sobolev inequality with constant $\lambda_{LS} > 0$ if:
\[
\frac{\lambda_{LS}}{2} \Ent_{\mu}(f^2) \leq \int |\nabla f|^2 d\mu \;\;\; \forall f \in \F_{loc} ~.
\]
Here $\Ent_\mu(g) := \int g \log \brac{g / \int g d\mu} d\mu$ denotes the entropy of a non-negative function $g$. 
\end{definition} 

For example, the space $(\R^n,|\cdot|,\gamma_n)$, where $\gamma_n$ denotes the standard $n$-dimensional Gaussian measure and $|\cdot|$ is the Euclidean metric, satisfies the above inequalities with sharp constants $\lambda_1 = \lambda_{LS} = 1$. In general, a log-Sobolev inequality always implies a Poincar\'e inequality with $\lambda_{1} \geq \lambda_{LS}$, but not vice versa \cite{Ledoux-Book,BGL-Book}.

As mentioned in the Introduction, the log-Sobolev inequality ensures by the Herbst argument that (\ref{eq:intro-RH}) holds for all $q < p$ (see also \cite{AidaMasudaShigekawa}). The argument is based on the classical observation that:
\begin{equation} \label{eq:identity}
\frac{d}{dt} \log \brac{\int |f|^t d\mu}^{\frac{1}{t}} = \frac{1}{t^2} \frac{\Ent_{\mu}(|f|^t)}{\int |f|^t d\mu} . 
\end{equation}
Applying this to $f = \exp(g)$ with $|\nabla g| \leq L$, invoking the log-Sobolev inequality, and integrating in $t$ from $q$ to $p$, (\ref{eq:intro-RH}) easily follows. We illustrate this by repeating the same argument for the \emph{modified} log-Sobolev inequality of Bobkov--Ledoux \cite{BobkovLedouxModLogSobAndPoincare}. We refer to \cite{EMilmanGeometricApproachPartII} for additional background on modified log-Sobolev inequalities. 

\begin{proof}[Proof of Theorem \ref{thm:intro-Poincare}]
It was shown by Bobkov--Ledoux in \cite[Theorem 3.1]{BobkovLedouxModLogSobAndPoincare} that if $(\Omega,d,\mu)$ satisfies a Poincar\'e inequality with constant $\lambda_1>0$ then for any bounded Borel function $g : (\Omega,d) \rightarrow \R$ with $|\nabla g| \leq \ell  < 2 \sqrt{\lambda_1}$, one has:
\begin{equation} \label{eq:modified-LS}
\Ent_\mu(e^g) \leq K(\ell) \int |\nabla g|^2 e^g d\mu ,
\end{equation}
with $K(\ell)$ given by (\ref{eq:K}).

Let $f$ be an $L$-log-Lipschitz function on $(\Omega,d)$ with $\eps \leq f \leq 1/\eps$ and let $0 < t < 2 \sqrt{\lambda_1}/L$. It follows by (\ref{eq:identity}) and (\ref{eq:modified-LS}) applied to $g = t \log f$ that:
\[
\frac{d}{dt} \log \brac{\int f^t d\mu}^{\frac{1}{t}} \leq L^2 K(tL) . 
\]
Integrating this inequality in $t$ from $q$ to $p$, the assertion of the Theorem follows for all $0 < q < p < 2 \sqrt{\lambda_1}/L$.  Since $K$ is integrable at the origin, the assertion extends to $q=0$. Since $\norm{f^{-1}}_{L^q(\mu)} = \norm{f}_{L^{-q}(\mu)}^{-1}$, by applying the above to $f^{-1}$ we obtain the assertion in the negative range as well. Finally, the restriction that $\eps \leq f \leq 1/\eps$ is removed by a trivial approximation argument. 
\end{proof}

\subsection{Infimum-Convolution Inequalities}

The most convenient and transparent way to formulate our  condition for ensuring (\ref{eq:intro-RH0}) is by using infimum-convolution. 
Let $c : \Omega \times \Omega \rightarrow \R_+$ denote a non-negative cost-function which we assume is upper semi-continuous, and satisfies $c(x,y) \leq a(x) + b(y)$ for some Borel functions $a,b : \Omega \rightarrow \R$. The infimum-convolution of a Borel function $f$ with respect to the cost-function $c$ is defined as:
\[
Q_c f (x) := \inf_{y \in \Omega} \brac{ f(y) + c(x,y) }. 
\]
Note that when $f$ is $1$-Lipschitz then $Q_d f = f$. 

Let $\varphi,\Phi : \R_+ \rightarrow \R$ denote two convex non-decreasing (and hence continuous) functions; we assume that $\varphi$ is non-negative with $\varphi(0) =0$, but do not make such a restriction on $\Phi$. Recall \cite[Chapter X]{ConvexAnalysisBookII} that the Legendre-Fenchel transform $\Psi^* : \R_+ \rightarrow \R \cup \{+\infty\}$ of $\Psi \in \{\varphi,\Phi\}$  is defined as:
\[
\Psi^*(s) = \sup_{t > 0} \{ s \cdot t - \Psi(t) \},
\]
and that under our assumptions, $(\Psi^*)^* = \Psi$. Note that $\Psi^*(s)$ is finite for all $s \geq 0$ if and only if $\lim_{t \rightarrow \infty} \Psi(t) / t = +\infty$. 
We denote by $c_\varphi$ the cost-function obtained by composing $\varphi \circ d$. 

\begin{definition}
$(\Omega,d,\mu)$ is said to satisfy a $(\varphi,\Phi)$-Infimum-Convolution inequality if for any $\mu$-integrable function $f : \Omega \rightarrow \R$:
\begin{equation} \label{eq:IC}
\int \exp(\lambda Q_{c_\varphi} f) d\mu \leq \exp \brac{ \lambda \int f d\mu + \Phi^*(\lambda)} \;\;\; \forall \lambda > 0 . 
\end{equation}
When $\Phi(0) = 0$ (or equivalently $\Phi^*(0) = 0$), we will say that the inequality is tight. 
\end{definition}

When $\varphi = \Id$ so that $c_\varphi = d$, since  $Q_d f = f$ for all $1$-Lipschitz functions, the $(\Id,\Phi)$-Infimum-Convolution inequality boils down to controlling the Laplace transform of $f$:
\[
\forall \text{$1$-Lipschitz $f$ with $\int f d\mu = 0$} \;\;, \;\; \int \exp(\lambda f) d\mu \leq \exp( \Phi^*(\lambda)) \;\;\; \forall \lambda > 0 . 
\]
By the Markov-Chebyshev inequality, this entails the following sub-exponential concentration inequality of $1$-Lipschitz functions $f$ about their expected value:
\begin{align}
\nonumber & \mu \{ x \in \Omega \; ; \; f(x) \geq \int f d\mu + t \} \leq \inf_{\lambda > 0} \frac{\int \exp(\lambda (f - \int f d\mu)) d\mu}{\exp(\lambda t)}\\
\label{eq:concentration} &  \leq \inf_{\lambda > 0} \exp(\Phi^*(\lambda) - \lambda t) = \exp(-\Phi(t)) \;\;\; \forall t > 0 . 
\end{align}
Conversely, it is not hard to show that whenever a sub-exponential concentration inequality of the above form holds, it may be integrated back to an inequality for the Laplace transform (after an adjustment of constants in various places - see \cite[Lemma 4.2]{EMilmanGeometricApproachPartII} for a precise statement). Note that since we allow $\Phi(0)$ to be negative, (\ref{eq:concentration}) indeed corresponds to a concentration inequality for the \emph{large-deviation} of $1$-Lipschitz functions when $t \gg 1$, and need not provide any information when $t \rightarrow 0$. However, when $\Phi(0) = 0$, we say that the resulting concentration inequality is tight. 

Denoting $S := \lim_{t \rightarrow \infty} \Phi(t) / t $, when $S< \infty$ we say that we have \emph{exponential concentration} of Lipschitz functions, whereas if $S = \infty$ the concentration is said to be \emph{strictly sub-exponential}; similarly, if $\liminf_{t \rightarrow \infty} \Phi(t) / t^2 > 0$ the concentration is said to be \emph{sub-Gaussian}. The Herbst argument easily shows that a log-Sobolev inequality implies sub-Gaussian concentration, and that a Poincar\'e inequality implies sub-exponential concentration \cite{Ledoux-Book,BGL-Book}. 

Summarizing, we see that $(\Id,\Phi)$-Infimum-Convolution inequalities are equivalent to sub-exponential concentration inequalities for $1$-Lipschitz functions. By using more general convex functions $\varphi$, $(\varphi,\Phi)$-Infimum-Convolution inequalities yield stronger conditions on the space. It will be more convenient to describe those using the equivalent Transport-Entropy inequalities, described next.

\subsection{Transport-Entropy Inequalities}

Transport-Entropy inequalities were first introduced by Marton \cite{Marton86,Marton96} and significantly developed by Talagrand \cite{TalagrandT2}.
These compare between the cost of optimally transporting between $\mu$ and a second Borel probability measure $\nu$  and the relative entropy of $\nu$ with respect to $\mu$.
The transport cost, or Wasserstein distance, between two Borel probability measures $\nu_1,\nu_2$ on $(\Omega,d)$,  is defined as:
\[
W_c(\nu_1,\nu_2) := \inf \int_{\Omega \times \Omega} c(x,y) d\Pi(x,y) ,
\]
where the infimum is taken over all probability measures $\Pi$ on the
product space $\Omega \times \Omega$ with marginals $\nu_1$ and $\nu_2$, respectively. 
The relative entropy, or Kullback--Leibler divergence with respect to $\mu$, is defined for $\nu \ll \mu$ as:
\[
H(\nu | \mu) := \Ent_\mu \brac{\frac{d\nu}{d\mu}} = \int \log\brac{\frac{d\nu}{d\mu}} d\nu ~,
\]
and $+\infty$ otherwise. 

\begin{definition}
$(\Omega,d,\mu)$ is said to satisfy a $(\varphi,\Phi)$-Transport-Entropy inequality if for any Borel probability measure $\nu$ on $(\Omega,d)$:
\[
W_{c_\varphi}(\nu , \mu) \leq \Phi^{-1}(H(\nu | \mu) ) . 
\]
\end{definition}
\noindent Here we use the convention $\Phi^{-1}(y) := \sup\{t \geq 0 \; ; \; \Phi(t) \leq y \}$. 

A general form of the Monge--Kantorovich--Rubinstein dual representation of $W_c$ states that \cite[p. 120]{Ledoux-Book}:
\[
W_c(\nu,\mu) = \sup \brac{\int g d\nu - \int f d\mu} ,
\]
where the supremum is taken over all $\nu-$ and $\mu$-integrable functions $g$ and $f$, respectively, such that $g(x) \leq f(y) + c(x,y)$. Consequently, given $f$, it is always optimal to set $g = Q_c f$ above. On the basis of this, it was shown by Bobkov--G\"{o}tze \cite{BobkovGotzeLogSobolev} for the case of $c=d$ and extended by Bobkov--Gentil--Ledoux \cite{BobkovGentilLedoux} (cf. \cite[Proposition 6.2]{Ledoux-Book}) to the general case that Transport-Entropy inequalities are equivalent to Infimum-Convolution inequalities. We refer to \cite[Theorem 4.1]{EMilmanGeometricApproachPartII} for the most general formulation below.

\begin{proposition} \label{prop:equivalent}
$(\Omega,d,\mu)$ satisfies a $(\varphi,\Phi)$-Transport-Entropy inequality if and only if it satisfies a $(\varphi,\Phi)$-Infimum-Convolution inequality. 
\end{proposition}

Note that by Jensen's inequality:
\[
\varphi(W_{d}(\nu,\mu)) \leq W_{c_\varphi}(\nu,\mu) .
\]
Consequently, a $(\varphi,\Phi)$-Transport-Entropy inequality immediately implies the weaker $(\Id,\Phi \circ \varphi)$ concentration inequality, and so by Proposition \ref{prop:equivalent}, the same holds true for Infimum-Convolution inequalities. It is also possible to see this implication directly on the level of Infimum-Convolution inequalities, by applying the $(\varphi,\Phi)$-Infimum-Convolution inequality to $\frac{\lambda}{\alpha} f$:
\[
\exp(-\alpha \varphi^*(\lambda/ \alpha)) \int \exp (\lambda Q_d f) d\mu \leq \int \exp(\alpha Q_{c_\varphi} \frac{\lambda}{\alpha} f) d\mu  \leq \exp \brac{ \alpha \int \frac{\lambda}{\alpha} f d\mu + \Phi^*(\alpha)} ,
\]
optimizing on $\alpha > 0$, and recalling that  under our assumptions \cite[Chapter X, Section 2.5]{ConvexAnalysisBookII}:
\begin{equation} \label{eq:composition}
(\Phi \circ \varphi)^*(\lambda) = \inf_{\alpha > 0} \Phi^*(\alpha) + (\alpha \varphi)^*(\lambda) = \inf_{\alpha > 0} \Phi^*(\alpha) + \alpha \varphi^*(\lambda/\alpha) ,
\end{equation}
which yields:
\[
\int \exp(\lambda Q_d f) d\mu \leq \exp \brac{ \lambda \int  f d\mu + (\Phi \circ \varphi)^*(\lambda)} .
\]

Let us recall several well-known inequalities, strictly stronger than their concentration counterparts, which are obtained by using various natural convex functions $\varphi$:

\begin{itemize}
\item Setting $\varphi(t) = \frac{\lambda_{T_2}}{2} t^2$ and $\Phi = \Id$ yields Talagrand's $T_2$ Transport-Entropy inequality \cite{TalagrandT2} (cf. \cite[Section 6.2]{Ledoux-Book}). It was shown by Otto--Villani \cite{OttoVillaniHWI} and Bobkov--Gentil--Ledoux \cite{BobkovGentilLedoux} in the smooth setting (see also Lott-Villani \cite{LottVillaniHamiltonJacobi} for extensions to the  general one) that a log-Sobolev inequality implies a $T_2$ inequality with $\lambda_{T_2} \geq \lambda_{LS}$. On the other hand, it was shown by Cattiaux--Guillin \cite{CattiauxGuillinT2InqWeakerThanLogSob} that a $T_2$ inequality does not imply back the log-Sobolev inequality in general. The preceding remarks show that a $T_2$ inequality implies a $(\Id, \frac{\lambda_{T_2}}{2} t^2)$-Infimum-Convolution inequality, i.e. sub-Gaussian concentration. 
\item Let $\varphi_1(t) := \min(t^2 /2 , t - 1/2)$. Setting $\varphi(t) = \varphi_1(\sqrt{\lambda_{T_1}} t)$ and $\Phi = \Id$, one obtains Talagrand's $T_1$ Transport-Entropy inequality \cite{TalagrandT2}. It was shown by Bobkov--Gentil--Ledoux \cite[Corollary 5.1]{BobkovGentilLedoux} that this inequality is equivalent to a Poincar\'e inequality, in the sense that $c_1 \lambda_1 \leq \lambda_{T_1} \leq c_2 \lambda_1$ for some universal numeric constants $c_1,c_2 > 0$, if $\lambda_{T_1}$ and $\lambda_1$ denote the best constants in the $T_1$ and Poincar\'e inequalities, respectively. 
\end{itemize}

\section{Statements} \label{sec:results}

We are now ready to formulate and prove our general conditions for ensuring (\ref{eq:intro-RH0}). 

\begin{theorem} \label{thm:main}
Assume that $(\Omega,d,\mu)$ satisfies a $(\varphi,\Phi)$ Transport-Entropy inequality, or equivalently, Infimum-Convolution inequality, and denote $\Psi := \Phi \circ \varphi$. Then for any $L$-log-Lipschitz function $f : (\Omega,d) \rightarrow \R_+$ and $p > 0$:
\[
C_{L,p,\Psi}^{-1} \norm{f}_{L^p(\mu)} \leq  \norm{f}_{L^0(\mu)} \leq C_{L,p,\Psi} \norm{f}_{L^{-p}(\mu)} ,
\]
with:
\[
C_{L,p,\psi} = \exp \brac{\Psi^*(pL)/p}  \in [1,\infty] . 
\]
In particular, denoting $S := \lim_{t \rightarrow \infty} \Psi(t) / t$:
\begin{enumerate}
\item Whenever $S = +\infty$ then $C_{L,p,\Psi} < \infty$ for all $p > 0$ and hence all $p$-moments of any log-Lipschitz function are comparable. 
\item If $S < \infty$ then $C_{L,p,\Psi} < \infty$ for all $p \in (0,S / L)$, and hence the $p$-moments of an $L$-log-Lipschitz function are comparable in the range $p \in (-S/L,S/L)$. 
\end{enumerate}
\end{theorem}

We will compare the borderline case when $S < \infty$ to the statement of Theorem \ref{thm:intro-Poincare} below. 
In the meantime, note that when $\Psi(t) = \frac{\lambda_{2}}{2} t^2$ corresponding to a tight sub-Gaussian concentration inequality with constant $\lambda_2$, we have $\Psi^*(s) = \frac{1}{2 \lambda_2} s^2$, so Theorem \ref{thm:main} yields (\ref{eq:intro-RH0}) from the Introduction with precisely the same dependence of $C_{L,p}$ on $L$ and $p$.

\begin{proof}[Proof of Theorem \ref{thm:main}]
The proof is immediate if we use the Infimum-Convolution formulation. By the remarks from the previous section, a $(\varphi,\Phi)$-Infimum-Convolution implies a $(\Id,\Psi)$-Infimum-Convolution (or concentration) inequality, so it is enough to treat this case, which is totally elementary. However, it may be insightful to treat the general $(\varphi,\Phi)$ case to verify that we do not get any improvement over the $(\Id,\Psi)$ case. 

Note that if $g$ is $L$-Lipschitz then:
\[
Q_{c_\varphi} g(x) = \inf_{y} g(y) + \varphi(d(x,y)) \geq \inf_{y} g(x) - L d(x,y) + \varphi(d(x,y)) \geq g(x) - \varphi^*(L) . 
\]
Applying this to the $\alpha L$-Lipschitz function $g = \alpha \log f$ and invoking (\ref{eq:IC}), we deduce:
\[
\int \exp(\lambda \alpha \log f ) d\mu \leq \exp \brac{ \lambda \alpha \int \log f d\mu + \Phi^*(\lambda) + \lambda \varphi^*(\alpha L) } \;\;\; \forall \lambda , \alpha > 0  .
\]
Consequently, if $p > 0$ is given, we may optimize on $\lambda > 0$ by setting $\alpha = p / \lambda$, yielding:
\[
\norm{f}_{L^p(\mu)} \leq \exp(\frac{1}{p} \inf_{\lambda >0} \brac{ \Phi^*(\lambda) + \lambda \varphi^*(p L / \lambda)}) \norm{f}_{L^0(\mu)} = \exp((\Phi \circ \varphi)^*(pL)/p) \norm{f}_{L^0(\mu)} ,
\]
where the last identify is due to (\ref{eq:composition}). 
Since $\norm{f^{-1}}_{L^q(\mu)} = \norm{f}_{L^{-q}(\mu)}^{-1}$, by applying the above to $f^{-1}$ we obtain the assertion for $-p$. 
\end{proof}

The above wrongly suggests that there is no advantage in using the more general $(\varphi,\Phi)$-Infimum-Convolution inequality over the weaker $(\Id,\Psi)$ one when it comes to obtaining reverse H\"{o}lder inequalities. However, the advantage does become apparent when the function is not assumed to be log-Lipschitz, but only so in some averaged sense (in fact, up to an additive term and only from one side).

\begin{theorem} \label{thm:Lb}
With the same assumptions and notation as in Theorem \ref{thm:main}, let $f : (\Omega,d) \rightarrow \R_+$ be a Borel function so that:
\begin{equation} \label{eq:cond}
\log f(y) \geq \log f(x) - L(x) d(x,y) - b(x) \;\;\; \forall x,y \in \Omega,
\end{equation}
for some Borel functions $L,b : (\Omega,d) \rightarrow \R_+$. Then for all $p > 0$:
\[
\norm{f}_{L^p(\mu)} \leq \inf_{\gamma > 1, \alpha > 0} \norm{ \exp \brac{ \frac{\Phi^*(p \gamma / \alpha)}{p \gamma} + \frac{p \gamma}{\alpha} (\varphi^*(\alpha L(x)) + \alpha b(x))} }_{L^\frac{1}{\gamma-1}(\mu)} \norm{f}_{L^0(\mu)} . 
\]
\end{theorem}
\begin{remark}
Note that the assumption (\ref{eq:cond}) is not symmetric in $x,y$, as the functions $L$ and $b$ depend only on $x$ and not on $y$; as a consequence, the conclusion only holds for $p > 0$.
 In particular, (\ref{eq:cond}) holds in the Euclidean setting with $L(x) = |\nabla \log f|(x)$ and $b \equiv 0$ if $\log f$ is convex. Compare with \cite[Corollary 6.1]{BobkovGentilLedoux}.
 \end{remark}

Note that we can longer invoke (\ref{eq:composition}) to optimize on $\alpha > 0$ inside the integral, since the optimal value of $\alpha$ will depend on $x$. In particular, if $\varphi = \Id$ (corresponding to having a concentration inequality), or more generally, if $\lim_{t \rightarrow \infty} \varphi(t) / t = \infty$, and $L(x)$ is not in $L^\infty(\mu)$, the above expression is necessarily infinite and we do not obtain any information. 

\begin{proof}[Proof of Theorem \ref{thm:Lb}]
Arguing as in the proof of Theorem \ref{thm:main}, we have:
\[
\int \exp \brac{\lambda \brac{\alpha \log f - \varphi^*(\alpha L(x)) - \alpha b(x)}} d\mu \leq \exp \brac{ \lambda \alpha \int \log f d\mu + \Phi^*(\lambda)} \;\;\; \forall \lambda , \alpha > 0  .
\]
Applying  H\"{o}lder's inequality, we deduce for all $\beta \in (0,1)$:
\[
\brac{\int \exp( \beta \lambda \alpha \log f) d\mu}^{1/\beta} \leq \brac{\int \exp( - \frac{\beta}{\beta-1} \lambda (\varphi^*(\alpha L(x)) + \alpha b(x))) d\mu}^{- \frac{\beta-1}{\beta}} \exp(  \lambda \alpha \int \log f d\mu + \Phi^*(\lambda)) . 
\]
Taking the $\lambda \alpha$-th root, setting $p = \beta \lambda \alpha$, $\gamma = \frac{1}{\beta} \in (1,\infty)$ and optimizing on $\alpha,\gamma$, the asserted estimate follows. 
\end{proof}

To conclude our discussion, let us turn our attention to the borderline case when $\lim_{t \rightarrow \infty} \Psi(t) / t < \infty$, in which case we can only expect an $L$-log-Lipschitz function $f$ to have comparable $p$-moments for a bounded range of $p$'s. As usual, we assume that a $(\varphi,\Phi)$-Transport-Entropy inequality holds, and denote $\Psi = \Phi \circ \varphi$. 

The weakest information is given by a non-tight exponential concentration inequality, corresponding to the case $\varphi = \Id$ and $\Phi(t) = -M + \lambda_{\exp} t$ with $M > 0$. In that case, we obtain by Theorem \ref{thm:main}:
\[
\exp(- M) \norm{f}_{L^{p}(\mu)} \leq \norm{f}_{L^0(\mu)} \leq \exp(M) \norm{f}_{L^{-p}(\mu)} \;\;\; \forall p \in [0,\lambda_{\exp}/L] . 
\]
This has the drawback that the constant $\exp(M)$ does not tend to $1$ as $p \rightarrow 0$. 

Stronger information is given when $\varphi = \varphi_1$, in which case the $(\varphi_1,\Id)$-Transport-Entropy (or Infimum-Convolution) inequality is equivalent to a Poincar\'e inequality. The fact that a Poincar\'e inequality always yields exponential concentration was first observed by Gromov--V. Milman \cite{GromovMilmanLevyFamilies}; using our notation, they showed that for all $1$-Lipschitz functions $g$ with $\int g d\mu = 0$  one has:
\[
\int \exp(\lambda g) d\mu \leq F(\lambda , \lambda_1) < \infty \;\;\;\; \forall \lambda \in (0, c \sqrt{\lambda_1}) , 
\]
for some explicit function $F$ and universal constant $c > 0$. The sharp constant $c =2$ (as witnessed by the double exponential measure $(\R,|\cdot|,\nu)$ for which $\lambda_1 = \frac{1}{4}$) was first obtained by Schmuckenschl\"{a}ger \cite{Schmucky-PoincareAndMartingales}. The value of the function $F$ was further sharpened by Bobkov--Ledoux \cite[Proposition 4.1]{BobkovLedouxModLogSobAndPoincare}, who showed that with the same assumptions as above:
\[
\int \exp(\lambda g) d\mu \leq \frac{2 \sqrt{\lambda_1} + \lambda}{2 \sqrt{\lambda_1} - \lambda}  \;\;\;\; \forall \lambda \in (0 , 2 \sqrt{\lambda_1}) 
\]
(in fact, a slightly better estimate is available when $g$ possesses a certain symmetry). 
Applying this to $g = (\log f - \int \log f d\mu)/ L$ for an $L$-log-Lipschitz function $f$, we immediately obtain:

\begin{theorem}[following Bobkov--Ledoux \cite{BobkovLedouxModLogSobAndPoincare}] \label{thm:Poincare}
Assume that $(\Omega,d,\mu)$ satisfies a Poincar\'e inequality with constant $\lambda_1 > 0$. 
Then for any $L$-log-Lipschitz function $f$:
\[
\brac{ \frac{2 \sqrt{\lambda_1} + p L}{2 \sqrt{\lambda_1} - p L} }^{-\frac{1}{p}} \norm{f}_{L^p(\mu)} \leq  \norm{f}_{L^0(\mu)}  \leq \brac{ \frac{2 \sqrt{\lambda_1} + p L}{2 \sqrt{\lambda_1} - p L} }^{\frac{1}{p}}  \norm{f}_{L^{-p}(\mu)} \;\;\; \forall p \in (0, 2 \sqrt{\lambda_1} / L) .
\]
\end{theorem}
This yields a more palatable bound than the one from Theorem \ref{thm:intro-Poincare} if one does not care about the property that  $\lim_{p \rightarrow q+} C_{L,q,p} \rightarrow 1$.

\setlinespacing{1.0}
\setlength{\bibspacing}{2pt}

\bibliographystyle{plain}
\bibliography{../../../ConvexBib}

\def\cprime{$'$} \def\textasciitilde{$\sim$}
\begin{thebibliography}{10}

\bibitem{AidaMasudaShigekawa}
S.~Aida, T.~Masuda, and I.~Shigekawa.
\newblock Logarithmic {S}obolev inequalities and exponential integrability.
\newblock {\em J. Funct. Anal.}, 126(1):83--101, 1994.

\bibitem{BGL-Book}
D.~Bakry, I.~Gentil, and M.~Ledoux.
\newblock {\em Analysis and geometry of {M}arkov diffusion operators}, volume
  348 of {\em Grundlehren der Mathematischen Wissenschaften [Fundamental
  Principles of Mathematical Sciences]}.
\newblock Springer, Cham, 2014.

\bibitem{BobkovPolynomials}
S.~G. Bobkov.
\newblock Remarks on the growth of {$L\sp p$}-norms of polynomials.
\newblock In {\em Geometric aspects of functional analysis}, volume 1745 of
  {\em Lecture Notes in Math.}, pages 27--35. Springer, Berlin, 2000.

\bibitem{BobkovGentilLedoux}
S.~G. Bobkov, I.~Gentil, and M.~Ledoux.
\newblock Hypercontractivity of {H}amilton-{J}acobi equations.
\newblock {\em J. Math. Pures Appl. (9)}, 80(7):669--696, 2001.

\bibitem{BobkovGotzeLogSobolev}
S.~G. Bobkov and F.~G{\"o}tze.
\newblock Exponential integrability and transportation cost related to
  logarithmic {S}obolev inequalities.
\newblock {\em J. Funct. Anal.}, 163(1):1--28, 1999.

\bibitem{BobkovHoudre}
S.~G. Bobkov and C.~Houdr{\'e}.
\newblock Isoperimetric constants for product probability measures.
\newblock {\em Ann. Probab.}, 25(1):184--205, 1997.

\bibitem{BobkovLedouxModLogSobAndPoincare}
S.~G. Bobkov and M.~Ledoux.
\newblock Poincar\'e's inequalities and {T}alagrand's concentration phenomenon
  for the exponential distribution.
\newblock {\em Probab. Theory Related Fields}, 107(3):383--400, 1997.

\bibitem{Bonami-LambdaPOnCube}
A.~Bonami.
\newblock Ensembles {$\Lambda (p)$} dans le dual de {$D^{\infty }$}.
\newblock {\em Ann. Inst. Fourier (Grenoble)}, 18(fasc., fasc. 2):193--204
  (1969), 1968.

\bibitem{Bonami-PhDInJournal}
A.~Bonami.
\newblock \'{E}tude des coefficients de {F}ourier des fonctions de
  {$L^{p}(G)$}.
\newblock {\em Ann. Inst. Fourier (Grenoble)}, 20(fasc., fasc. 2):335--402
  (1971), 1970.

\bibitem{Borell-logconcave}
Ch. Borell.
\newblock Convex measures on locally convex spaces.
\newblock {\em Ark. Mat.}, 12:239--252, 1974.

\bibitem{Bourgain-LK}
J.~Bourgain.
\newblock On the distribution of polynomials on high dimensional convex sets.
\newblock In {\em Geometric Aspects of Functional Analysis}, volume 1469 of
  {\em Lecture Notes in Math.}, pages 127--137. Springer-Verlag, 1991.

\bibitem{CarberyWrightPolynomials}
A.~Carbery and J.~Wright.
\newblock Distributional and {$L^q$} norm inequalities for polynomials over
  convex bodies in {$\Bbb R^n$}.
\newblock {\em Math. Res. Lett.}, 8(3):233--248, 2001.

\bibitem{CattiauxGuillinT2InqWeakerThanLogSob}
P.~Cattiaux and A.~Guillin.
\newblock On quadratic transportation cost inequalities.
\newblock {\em J. Math. Pures Appl. (9)}, 86(4):341--361, 2006.

\bibitem{Chiti-ReverseHolderForDirichletEigenfunctions}
G.~Chiti.
\newblock A reverse {H}\"{o}lder inequality for the eigenfunctions of linear
  second order elliptic operators.
\newblock {\em Z. Angew. Math. Phys.}, 33(1):143--148, 1982.

\bibitem{CianchiMazya-ReverseHolderForEigenfunctions}
A.~Cianchi and V.~G. Maz'ya.
\newblock Bounds for eigenfunctions of the {L}aplacian on noncompact
  {R}iemannian manifolds.
\newblock {\em Amer. J. Math.}, 135(3):579--635, 2013.

\bibitem{ReverseHolderOnSphere}
J.~Duoandikoetxea.
\newblock Reverse {H}\"{o}lder inequalities for spherical harmonics.
\newblock {\em Proc. Amer. Math. Soc.}, 101(3):487--491, 1987.

\bibitem{FleuryImprovedThinShell}
B.~Fleury.
\newblock Concentration in a thin euclidean shell for log-concave measures.
\newblock {\em J. Func. Anal.}, 259:832--841, 2010.

\bibitem{FradeliziUltimateKhinchine}
M.~Fradelizi.
\newblock Concentration inequalities for {$s$}-concave measures of dilations of
  {B}orel sets and applications.
\newblock {\em Electron. J. Probab.}, 14:no. 71, 2068--2090, 2009.

\bibitem{GromovMilmanLevyFamilies}
M.~Gromov and V.~D. Milman.
\newblock A topological application of the isoperimetric inequality.
\newblock {\em Amer. J. Math.}, 105(4):843--854, 1983.

\bibitem{Guedon-extension-to-negative-p}
O.~Gu{\'{e}}don.
\newblock Kahane-khinchine type inequalities for negative exponent.
\newblock {\em Mathematika}, 46:165--173, 1999.

\bibitem{GuedonEMilmanInterpolating}
O.~Gu{\'{e}}don and E.~Milman.
\newblock Interpolating thin-shell and sharp large-deviation estimates for
  isotropic log-concave measures.
\newblock {\em Geom. Func. Anal.}, 21(5):1043--1068, 2011.

\bibitem{ConvexAnalysisBookII}
J.-B. Hiriart-Urruty and C.~Lemar\'{e}chal.
\newblock {\em Convex analysis and minimization algorithms. {II}}, volume 306
  of {\em Grundlehren der Mathematischen Wissenschaften [Fundamental Principles
  of Mathematical Sciences]}.
\newblock Springer-Verlag, Berlin, 1993.
\newblock Advanced theory and bundle methods.

\bibitem{KKL-Influence}
J.~Kahn, G.~Kalai, and N.~Linial.
\newblock The influence of variables on boolean functions.
\newblock In {\em Proceedings of the 29th Annual Symposium on Foundations of
  Computer Science}, SFCS '88, page 68–80, USA, 1988. IEEE Computer Society.

\bibitem{KlartagCLP}
B.~Klartag.
\newblock A central limit theorem for convex sets.
\newblock {\em Invent. Math.}, 168:91--131, 2007.

\bibitem{KlartagCLPpolynomial}
B.~Klartag.
\newblock Power-law estimates for the central limit theorem for convex sets.
\newblock {\em J. Funct. Anal.}, 245:284--310, 2007.

\bibitem{KlartagUnconditionalVariance}
B.~Klartag.
\newblock A {B}erry-{E}sseen type inequality for convex bodies with an
  unconditional basis.
\newblock {\em Probab. Theory Related Fields}, 45(1):1--33, 2009.

\bibitem{KohlerJobin-ReverseHolderForGroundState}
M.-T. Kohler-Jobin.
\newblock Isoperimetric monotonicity and isoperimetric inequalities of
  {P}ayne-{R}ayner type for the first eigenfunction of the {H}elmholtz problem.
\newblock {\em Z. Angew. Math. Phys.}, 32(6):625--646, 1981.

\bibitem{LatalaZeroMomentKhinchine}
R.~Lata{\l}a.
\newblock On the equivalence between geometric and arithmetic means for
  log-concave measures.
\newblock In {\em Convex geometric analysis ({B}erkeley, {CA}, 1996)},
  volume~34 of {\em Math. Sci. Res. Inst. Publ.}, pages 123--127. Cambridge
  Univ. Press, Cambridge, 1999.

\bibitem{Ledoux-Book}
M.~Ledoux.
\newblock {\em The concentration of measure phenomenon}, volume~89 of {\em
  Mathematical Surveys and Monographs}.
\newblock American Mathematical Society, Providence, RI, 2001.

\bibitem{LeeVempala-KLS}
Y.~T. Lee and S.~S. Vempala.
\newblock Eldan's stochastic localization and the {KLS} hyperplane conjecture:
  An improved lower bound for expansion.
\newblock In {\em 2017 IEEE 58th Annual Symposium on Foundations of Computer
  Science (FOCS)}, pages 998--1007, Oct 2017.

\bibitem{LottVillaniHamiltonJacobi}
J.~Lott and C.~Villani.
\newblock Hamilton-{J}acobi semigroup on length spaces and applications.
\newblock {\em J. Math. Pures Appl. (9)}, 88(3):219--229, 2007.

\bibitem{Marton86}
K.~Marton.
\newblock A simple proof of the blowing-up lemma.
\newblock {\em IEEE Trans. Inform. Theory}, 32(3):445--446, 1986.

\bibitem{Marton96}
K.~Marton.
\newblock Bounding {$\overline d$}-distance by informational divergence: a
  method to prove measure concentration.
\newblock {\em Ann. Probab.}, 24(2):857--866, 1996.

\bibitem{EMilmanGeometricApproachPartII}
E.~Milman.
\newblock Properties of isoperimetric, functional and transport-entropy
  inequalities via concentration.
\newblock {\em Probab. Theory Relat. Fields}, 152:475--507, 2012.

\bibitem{NazarovSodinVolbergPolynomials}
F.~Nazarov, M.~Sodin, and A.~Vol{\cprime}berg.
\newblock The geometric {K}annan-{L}ov\'asz-{S}imonovits lemma, dimension-free
  estimates for the distribution of the values of polynomials, and the
  distribution of the zeros of random analytic functions.
\newblock {\em Algebra i Analiz}, 14(2):214--234, 2002.

\bibitem{Nelson-Hypercontractivity}
E.~Nelson.
\newblock The free {M}arkoff field.
\newblock {\em J. Functional Analysis}, 12:211--227, 1973.

\bibitem{ODonnell-Book}
R.~O'Donnell.
\newblock {\em Analysis of {B}oolean functions}.
\newblock Cambridge University Press, New York, 2014.

\bibitem{OttoVillaniHWI}
F.~Otto and C.~Villani.
\newblock Generalization of an inequality by {T}alagrand and links with the
  logarithmic {S}obolev inequality.
\newblock {\em J. Funct. Anal.}, 173(2):361--400, 2000.

\bibitem{Paouris-IsotropicTail}
G.~Paouris.
\newblock Concentration of mass on convex bodies.
\newblock {\em Geom. Funct. Anal.}, 16(5):1021--1049, 2006.

\bibitem{PaourisSmallBall}
G.~Paouris.
\newblock Small ball probability estimates for log-concave measures.
\newblock {\em Trans. Amer. Math. Soc.}, 364(1):287--308, 2012.

\bibitem{PayneRayner1}
L.~E. Payne and M.~E. Rayner.
\newblock An isoperimetric inequality for the first eigenfunction in the fixed
  membrane problem.
\newblock {\em Z. Angew. Math. Phys.}, 23:13--15, 1972.

\bibitem{PayneRayner2}
L.~E. Payne and M.~E. Rayner.
\newblock Some isoperimetric norm bounds for solutions of the {H}elmholtz
  equation.
\newblock {\em Z. Angew. Math. Phys.}, 24:105--110, 1973.

\bibitem{Schmucky-PoincareAndMartingales}
M.~Schmuckenschl{\"a}ger.
\newblock Martingales, {P}oincar\'e type inequalities, and deviation
  inequalities.
\newblock {\em J. Funct. Anal.}, 155(2):303--323, 1998.

\bibitem{TalagrandT2}
M.~Talagrand.
\newblock Transportation cost for {G}aussian and other product measures.
\newblock {\em Geom. Funct. Anal.}, 6(3):587--600, 1996.

\end{thebibliography}

\end{document}